\documentclass[12pt,a4paper]{amsart}
\makeatletter
\renewcommand\normalsize{%
    \@setfontsize\normalsize{11.7}{14pt plus .3pt minus .3pt}%
    \abovedisplayskip 10\p@ \@plus4\p@ \@minus4\p@
    \abovedisplayshortskip 6\p@ \@plus2\p@
    \belowdisplayshortskip 6\p@ \@plus2\p@
    \belowdisplayskip \abovedisplayskip}
\renewcommand\small{%
    \@setfontsize\small{9.5}{12\p@ plus .2\p@ minus .2\p@}%
    \abovedisplayskip 8.5\p@ \@plus4\p@ \@minus1\p@
    \belowdisplayskip \abovedisplayskip
    \abovedisplayshortskip \abovedisplayskip
    \belowdisplayshortskip \abovedisplayskip}
\renewcommand\footnotesize{%
    \@setfontsize\footnotesize{8.5}{9.25\p@ plus .1pt minus .1pt}%%
    \abovedisplayskip 6\p@ \@plus4\p@ \@minus1\p@
    \belowdisplayskip \abovedisplayskip
    \abovedisplayshortskip \abovedisplayskip
    \belowdisplayshortskip \abovedisplayskip}
\ifdefined\pdfpagewidth
\else
\fi
\calclayout
\makeatother

\usepackage{graphicx}
\usepackage{xcolor}
\usepackage{mathtools}
\usepackage{amsmath,amsthm,amsfonts,amssymb,amscd,enumerate}
\usepackage{xy}
\xyoption{all}
\usepackage{comment}

\usepackage{palatino}
\usepackage[]{euler}
\usepackage{mathrsfs}

\usepackage[hidelinks]{hyperref}

\usepackage{MnSymbol}
\numberwithin{equation}{subsection}
\swapnumbers

\newtheorem{theorem}[equation]{Theorem}
\newtheorem*{theorem*}{Theorem}
\newtheorem*{lemma*}{Lemma}

\newtheorem{corollary}[equation]{Corollary}
\newtheorem{proposition}[equation]{Proposition}
\newtheorem*{proposition*}{Proposition}
\newtheorem{lemma}[equation]{Lemma}
\newtheorem*{conjecture*}{Conjecture}

\theoremstyle{definition}
\newtheorem{definition}[equation]{Definition}

\newtheorem{example}[equation]{Example}

\newtheorem{remark}[equation]{Remark}

\newcommand{\R}{\mathbb{R}}
\newcommand{\C}{\mathbb{C}}

\DeclareMathOperator{\Ad}{Ad}

\DeclareMathOperator{\Hom}{Hom}
\DeclareMathOperator{\End}{End}

\DeclareMathOperator{\Ind}{Ind}

\newcommand{\lie}{\mathfrak}
\usepackage{cancel} % for \cancel & \xcancel in math.

\newcommand{\cpt}{\operatorname{cpt}}
\newcommand{\cusp}{\operatorname{cusp}}
\newcommand{\cmc}{\operatorname{cmc}}
\newcommand{\Interior}{\operatorname{int}\Satake}

\newcommand{\circG}{{}^\circ\hspace{-1pt}G}
\newcommand{\circg}{{}^\circ\hspace{-1pt}g}

\newcommand{\Oshima}{\mathscr{M}}
\newcommand{\Groupoid}{\mathscr{G}}
\newcommand{\Subgroupoid}{\mathscr{H}}

\newcommand{\Satake}{\mathscr{X}}

\newcommand{\ldoublebracket}{[\![}
\newcommand{\rdoublebracket}{]\!]}
\newcommand{\Cc}{C_c}
\newcommand{\continuous}{continuous}

\newcommand{\bigG}{\mathbf{G}}
\newcommand{\bigH}{\mathbf{H}}
\newcommand{\nonnegative}{{\hspace{-.8pt}\circ\hspace{-.7pt}\plus}}
\newcommand{\spacenonnegative}{{\hspace{-.8pt}\circ\hspace{-.7pt}\plus}}

\begin{document}

\title[Lie groupoids and the tempered dual II]{Lie groupoids, the Satake compactification and the tempered dual, II: The Harish-Chandra principle}

\author{Jacob Bradd,  Nigel Higson and Robert Yuncken}

\date{}

\thanks{
This research was begun during the 2025 thematic trimester programme on Representation Theory and Noncommutative Geometry at the Institut Henri Poincar\'e. 
The authors acknowledge support of the Institut Henri Poincaré (UAR 839 CNRS-Sorbonne Université), and LabEx CARMIN (ANR-10-LABX-59-01).  This article is also based upon work from COST Action CaLISTA CA21109, supported by COST (European Cooperation in Science and Technology): \url{www.cost.eu}.  
J.~Bradd and R.~Yuncken were supported by project OpART of the Agence Nationale de la Recherche (ANR-23-CE40-0016) and R.~Yuncken was also suppported by project CroCQG (ANR-25-CE40-5010).  N.~Higson was supported  by the NSF grant DMS-1952669.  
}

\dedicatory{Dedicated to Georges Skandalis}

\begin{abstract}
We give a geometric account of Harish-Chandra's principle that a tempered irreducible representation of a real reductive group is either square-integrable modulo center, or embeddable in a representation that is parabolically induced from such a representation.  Our approach uses the Satake compactification,  an associated  groupoid that was constructed in the first paper of this series, and its $C^*$-algebra.
\end{abstract}

\maketitle

\section{Introduction} 
A basic organizing principle in the tempered representation theory of real reductive groups, discovered by Harish-Chandra, is that a tempered irreducible representation of a real reductive group is either square-integrable modulo center, or embeddable in a representation that is parabolically induced from a representation that is  square-integrable modulo center. This is one of the two foundational principles underpinning Harish-Chandra's   Plancherel formula, the other being that a reductive group possesses irreducible square-integrable representations precisely when it possesses a compact Cartan subgroup. 

The discrete series representations have been analyzed in detail from a geometric point of view---more precisely from an index-theoretic point of view---starting with the work of Parthasarathy \cite{Parthasarathy72} and Atiyah and Schmid \cite{AtiyahSchmid77}, and culminating in the work of Lafforgue \cite{Lafforgue02}. Harish-Chandra's first principle has received less attention.  

We shall use a $C^*$-algebra constructed from the {Satake compactification}   \cite{Satake60,BorelJiCompactificationsBook}   to give a conceptual, noncom\-mutative-geometric proof of Harish-Chandra's first principle. Our main results are stated precisely in Section~\ref{sec-statement-of-h-c-principle}.

$C^*$-algebras  play two roles in our argument. First,  Harish-Chandra's tempered irreducible unitary representations correspond precisely to those irred\-ucible unitary representations of $G$ that integrate to irreducible representations of the reduced group $C^*$-al\-gebra $C^*_r(G)$, and every irreducible representation of $C^*_r(G)$ is so-obtained \cite{CowlingHaagerupHowe88}.  Second,   $C^*$-algebra theory   provides a simple  tool to separate the space of all these irreducible representations into two parts: indeed if $A$ is any $C^*$-algebra, and if $J$ is any ideal in $A$, then there is a  partition the spectrum $\widehat A$ (the set of irreducible representations, up to equivalence) into those representations  that vanish on all elements of $J$, and those that don't, and this partition takes the simple form
\[
\widehat A = \widehat {A/J} \,\, \sqcup \,\, \widehat J;
\]
see for instance \cite[Sec.\,3.2]{DixmierCStarBook}.

In broad terms, glossing over some details for now, our argument will be as follows.  We shall introduce in Section~\ref{sec-topological-discrete-series} an ideal $I$ in $A{=}C^*_r(G)$ for which 
\[
\widehat I = \{ \, \text{discrete series representations of $G$} \,\}.
\]
This is a very general construction that may be applied to any unimodular locally compact group. Then, in Section~\ref{sec-parabolic-induction},  we shall define a second ideal $J \triangleleft A$ such that 
\[
\widehat {A/J} = \left \{ \, \parbox{3.4 in}{\begin{center}tempered irreducible representations of $G$ that embed in a principal series representation\end{center}} \,\right \}.
\]
The definition will be specific to real reductive groups, of course, but it will be otherwise very elementary, using only the definition of parabolic induction, as viewed from the perspective of $C^*$-algebra theory \cite{Clare:parabolic_induction,ClaCriHig:parabolic_induction}. Harish-Chandra's principle amounts to the assertion that $I=J$. 

In Section~\ref{sec-harish-chandra-principle} we shall prove that the ideals $I$ and $J$ coincide using the Satake compactification $\Satake$ of the symmetric space associated to $G$, and a  groupoid $\Groupoid_{\Satake}$ that we constructed and studied in \cite{BraddHigsonYunckenOshimaPart1}, following ideas of Omar Mohsen   \cite{Mohsen:blowup}. The reduced $C^*$-algebra of the groupoid, $C^*_r(\Groupoid_{\Satake})$ fits into an exact sequence 
\[
0 \longrightarrow C^*_r(\Groupoid_{\Interior }) \longrightarrow C^*_r(\Groupoid_{\Satake}) \longrightarrow C^*(\Groupoid_{\partial  \Satake}) \longrightarrow 0
\]
according to the decomposition of $\Satake$ into its interior and boundary.  We shall prove that $I{=}J$ by relating $I$ to the image of the inclusion morphism in the exact sequence,  and $J$ to the kernel of the quotient morphism; obviously the two ideals in the groupoid $C^*$-algebra are the same.  Crucial to the argument is a $C^*$-algebra morphism 
\[
C^*_r(G) \longrightarrow C^*_r(\Groupoid_{\Satake})
\]
that was introduced by Mohsen in \cite{Mohsen:blowup} and that is studied here first in Section~\ref{sec-satake-groupoid-and-c-star-algebra}, and then in a slightly generalized form in Section~\ref{sec-vector-bundles-on-satake}.

It  is a pleasure to thank Alain Connes and  Omar Mohsen  for illuminating conversations that  directly inspired our work.

\section{The topological discrete series}
\label{sec-topological-discrete-series}

In this section we shall investigate what might be called the topological discrete series of a  locally compact and unimodular  group; these are the irreducible unitary representations that are  isolated in the reduced unitary dual of the group.  Our treatment of the topic is  influenced by the approach  of Bernstein \cite{Bernstein92notes} to the smooth representation theory of $p$-adic groups.

For real reductive groups with compact center, the topological discrete series representations turn out to be the precisely  Harish-Chandra's discrete series (they are in any case obviously included in the discrete series). But  for our approach to the Harish-Chandra principle is is more appropriate to use the topological discrete series from the outset.

\subsection{The compact ideal in the reduced group C*-algebra} 
Throughout Section~\ref{sec-topological-discrete-series} we shall denote by $G$  a unimodular, locally compact and Hausdorff topological  group.

The following two definitions are appropriate only for groups with compact center (see Remark~\ref{rem-noncompact-g-problems}), but they can nevertheless be given for any $G$.

\begin{definition}
    An irreducible unitary representation $\pi$ of $G$ is a \emph{topological discrete series representation} if  the singleton set $\{ \pi\}$ is an open and closed subset of the reduced unitary dual of $G$ --- which is by definition the spectrum of the reduced group $C^*$-algebra $C^*_r (G)$.
\end{definition}

For the rest of this section, fix a compact subgroup $K \subseteq G$. The definition below depends on the choice of $K$, or rather on the conjugacy class of $K$.  But we shall avoid inserting $K$ into the terminology because when we turn to reductive groups it will be natural to pick a maximal compact subgroup, obviating the need to mention $K$ explicitly.

Let $V$ be (the underlying finite-dimensional Hilbert space of) a finite-dim\-en\-sional unitary representation of $K$. We shall be interested in the Hilbert space 
\begin{equation}
    \label{eq-def-of-l-2-g-mod-k-v}
L^2(G/K;V) = \bigl [ L^2 (G)\otimes V \bigr ]^K,
\end{equation}
where the fixed-point space is defined using the right-translation action of $K$ on $G$, and of course the given action of $K$ on $V$. More geometrically, this is the space of $L^2$-sections of the vector bundle over $G/K$ that is induced from $V$. The group $G$ acts on $L^2(G/K;V)$ via the left-translation action of $G$ on itself.

\begin{definition}
\label{def-compact-ideal}
    The \emph{compact ideal} $C^*_{\cpt}(G)\triangleleft C^*_r(G)$ is 
    \[
    \Bigl \{ \,
    f \in C_r^*(G) : \parbox{270pt}{\begin{center}$f$ acts as a compact operator on $L^2(G/K;V)$, for every \\ finite-dimensional unitary representation $V$ of $K$\end{center}}
    \, \Bigr \} 
    \]
\end{definition}

Our goal is to determine the relationship between the topological discrete series and the compact ideal.  For this we need one more definition.

\begin{definition} 
Let $G$ be a locally compact group and let $K$ be a compact subgroup of $G$.  A unitary representation $\pi$ of $G$ is \emph{$K$-admissible} if the restriction of $\pi$ to $K$ includes each irreducible representation of $K$ with at most finite multiplicity. 
\end{definition}

Recall  now that if $I$ is any closed, two-sided ideal in a $C^*$-algebra $A$, then every irreducible representation of $I$ extends in a unique way to an irreducible representation of $A$, and in this way the spectrum of $I$ is identified with an open subset of the spectrum of $A$, namely the open set of those irreducible representations of $A$ that are non-zero on $I$.  With this, the main result of Section~\ref{sec-topological-discrete-series} is as follows: 

\begin{theorem}\label{thm:Spec_J}
Fix $G$ and $K$ as above. The spectrum of the ideal $C^*_{\mathrm{cpt}}(G)$ consists of all \textup{(}equivalence classes of\textup{)} irreducible unitary representations of $G$  such that
    \begin{enumerate}[\rm (i)]
    
    \item $\pi$ is a topological discrete series representation, and 

    \item $\pi$ is $K$-admissible.
    
\end{enumerate}
\end{theorem}

To begin we shall show that every $K$-admissible topological discrete series representation belongs to the spectrum of the compact ideal (the converse requires a bit more work).  

For computations, it will be convenient to recast Definition~\ref{def-compact-ideal} as follows.  Corresponding to each irreducible unitary representation $V$ of $K$ there is a central projection $p\in C^*(K)$, which acts as the projection onto the $V$-isotypical subspace in any unitary representation of $K$, and every central projection in $C^*(K)$ is an orthogonal sum of such.  Since, when $V$ is irreducible,
\[
L^2(G){\cdot} p \cong \bigl [ L^2 (G)\otimes V^* \bigr ]^K\otimes V =L^2(G/K;V^*)\otimes V ,
\] 
where here we are using the right action of $C^*(K)$ on $L^2(G)$, the compact ideal  can be described as the ideal of those elements in $C^*_r(G)$ that act as compact operators on $L^2(G){\cdot} p$, for every central projection $p\in C^*(K)$.

\begin{lemma} 
\label{lem-admissible-tds-implies-square-integrable}
If an irreducible unitary representation $\pi$ is a $K$-admissible topological discrete series representation, then $\pi$ is square-integrable.
\end{lemma} 

\begin{proof} 
The union of all subalgebras 
$p{\cdot} C^*_r(G){\cdot} p \subseteq C^*_r(G)$,
as $p$ ranges over the central projections in $C^*(K)$, acting  as multipliers of $C^*_r(G)$, is a dense subalgebra of $C^*_r(G)$.  Under the assumption of $K$-admissibility, each $\pi (p{\cdot} C^*_r(G){\cdot} p)$ is a $C^*$-algebra of finite-rank operators on $H_\pi$. Therefore, assuming that $\pi$ is irreducible,  
\[
\pi (C^*_r (G))=\mathfrak{K}(H_\pi),
\]
the $C^*$-algebra of compact operators on $H_\pi$.

Now if $A$ is any  $C^*$-algebra, and if $\pi$ is an irreducible representation of $A$ such that $\{\pi\}$ is an open and closed subset of $\widehat A$, then $A$ decomposes as a direct sum of closed, two-sided ideals $A = I(\pi) \oplus \operatorname{ker}(\pi)$, where 
\begin{equation}
\label{eq-annihilator-ideal}
I(\pi) = \{\, a\in A  : a \cdot \operatorname{ker}(\pi) = \operatorname{ker}(\pi)\cdot a = 0\,\} ,
\end{equation}
and obviously the restriction of $\pi$ to $I(\pi)$ is an isomorphism onto $\pi ( A)$. Thus 
\[
A = I(\pi)\oplus \operatorname{ker}(\pi)\cong \pi(A) \oplus \operatorname{ker}(\pi).
\]
In the present case, because of the direct sum decomposition above, the closed subspace 
\[
\overline{I(\pi)\cdot L^2 (G)} \subseteq L^2 (G)
\]
is both a subrepresentation of the regular representation of $G$ and a representation of $I(\pi)\cong \mathfrak{K}(H_\pi)$.  As a representation of the compact operators, it is a direct sum of irreducible subrepresentations.  Each subrepresentation is also a representation of $G$, equivalent to $\pi$, since the action of $G$ factors through the projection $C^*_r(G)\to I(\pi)$.  So $\pi$ is embeddable into $L^2(G)$, and is therefore square-integrable; see for instance  \cite[Lem.~14.1.2]{DixmierCStarBook}.
\end{proof}

\begin{remark}
\label{rem-noncompact-g-problems}
    If $G$ fails to have compact center, then it has no square-integrable irreducible unitary representations, and therefore no $K$-admissible discrete series representations.  In addition, if $G$ fails to have compact center, then   $C^*_{\mathrm{cpt}}(G){=}0$.  Indeed, if $G$ fails to have compact center, then  the representation of the center   on $[L(G)\otimes V]^K$ has no finite-dimensional subrepresentations, whereas each nonzero self-adjoint element of $C^*_{\mathrm{cpt}}(G)$ would have a finite-dimensional eigenspace on some $[L(G)\otimes V]^K$ that would be such a finite-dimensional subrepresentation. So in the noncompact center case, Theorem~\ref{thm:Spec_J} is correct but vacuous.
\end{remark}

\begin{lemma}
    Every  $K$-admissible topological discrete series representation is included in the spectrum of the compact ideal.
\end{lemma}

\begin{proof}
    Let $\pi$ be a $K$-admissible topological discrete series representation of $G$.  By Lemma~\ref{lem-admissible-tds-implies-square-integrable} it is square-integrable, and therefore every subrepresentation of $L^2(G)$ that is equivalent to $\pi$ has the form
    \begin{equation}
        \label{eq-embedding-sq-int-rep-into-regular-rep}
    H_w = \{\, \varphi_{ w,v} \in L^2(G) : v\in H_\pi\,\}
    \end{equation}
    for some nonzero $w\in H_\pi$, where $\varphi_{ w,v}$ is the square-integrable matrix coefficient function 
    \begin{equation}
        \label{eq-matrix-coefficient-function}
    \varphi_{w,v}(g) = \langle w,\pi(g)^{-1} v\rangle_{H_\pi}.
    \end{equation}
    It follows from the admissibility of $\pi$ that for every finite-dimensional unitary representation $V$ of $K$ the space 
    \[
    \overline{I(\pi)\cdot L^2 (G/K;V)}   = \bigl [\overline{I(\pi)\cdot L^2 (G)}\otimes V \bigr ]^K ,
    \]
    with $I(\pi)$ the annihilator ideal in \eqref{eq-annihilator-ideal}, is a finite-direct sum of copies of $\pi$, and that therefore $I(\pi)$ acts compactly on it.  Therefore $I(\pi)\subseteq C^*_{\mathrm{cpt}}(G)$.
\end{proof}

\subsection{A spectral radius-type formula} In order to complete the proof of Theorem~\ref{thm:Spec_J} we need  the following simple but remarkable formula due to Cowling, Haagerup and Howe.

\begin{theorem}[{\cite[Remark (d), p.103]{CowlingHaagerupHowe88}}]
\label{thm-chh-formula}
   Let $G$ be any locally compact Hausdorff topological group. If $f\in C_c(G)$ then
    \[
    \|f\|_{C^*_r(G)} = \lim_{n\to\infty} \| (f^*f)^{n} \|_{L^2(G)}^{{1}/{2n}},
    \]
where the multiplication in $C_c(G)$ is  the convolution product, defined using a choice of Haar measure \textup{(}we use the same choice for both   sides of the formula\textup{)}, and the adjoint is the convolution adjoint.
\end{theorem}

\begin{lemma}
\label{lem:I_S_faithful}
Let $p\in C^*(K)$ be a central projection.
The ideal 
$C_r^*(G) p C_r^*(G)\triangleleft C^*_r(G) $
 acts isometrically on $L^2(G){\cdot} p$.
\end{lemma}

\begin{proof}[Proof of Lemma \ref{lem:I_S_faithful}] 
Let us write $\lambda:C^*_r(G) \to \mathfrak{B}(L^2 (G){\cdot} p)$ for the left regular representation of $C^*_r(G)$ on $L^2 (G) {\cdot} p$.   It suffices to show that
\begin{equation}
    \label{eq-isometric-rep}
  \|\lambda(a)\|_{\mathfrak{B}(L^2(G){\cdot} p)} \geq \|a\|_{C^*_r(G)}
\end{equation}
for $a\in C_r^*(G) p C_r^*(G)$, since the reverse inequality is automatic.  Moreover, it suffices to prove this for a dense subset of elements, so let us consider  elements of the form 
\[
  a = \sum_{j=1}^m x_jpy_j\in C_r^*(G){\cdot} p {\cdot} C_r^*(G),
\]
with $x_j,y_j \in C_c(G)$.  
If we denote by  $\rho$   the right action of $C^*_r (G)$ on $L^2(G)$, then 
\begin{align*}
  \|(a^*a)^{n}\|_{L^2(G)}
    &\leq \sum_{j=1}^m \| (a^*a)^{n-1} a^* x_j p y_j\|_{L^2(G)} \\
    &\leq \sum_{j=1}^m \| (a^*a)^{n-1} a^* x_j p \|_{L^2(G)}\, \|\rho(y_j)\|_{\mathfrak{B}(L^2(G))} \\
    &\leq \| \lambda(a) \|^{2n-1}_{\mathfrak{B}(L^2(G){\cdot} p)} \sum_{j=1}^m \|x_jp\|_{L^2(G){\cdot} p} \,\|\rho(y_j)\|_{\mathfrak{B}(L^2(G))}.
\end{align*}
Taking  $2n^{\mathrm{th}}$ roots, and then taking the limit as $n$ tends to infinity,  we obtain \eqref{eq-isometric-rep} from Theorem~\ref{thm-chh-formula}.
\end{proof}

\subsection{Spectrum of the compact ideal}
We can now complete the proof of Theorem~\ref{thm:Spec_J}:

\begin{lemma}
\label{lem-spectrum-of-compact-ideal-is-only-sq-int-reps}
    Every irreducible unitary representation in the spectrum of the compact ideal, when viewed as an irreducible unitary representation of $G$, is a square-integrable representation.
\end{lemma}

\begin{proof} 
For each central projection $p\in C^*(K)$ write  
\begin{equation}
    \label{eq-def-of-j-p}
J_p =   \overline{C_{\cpt}^*(G){\cdot}p{\cdot} C_{\cpt}^*(G)}.
\end{equation}
This is a closed, two-sided ideal in $C^*_{\cpt}(G)$ and in $C^*_r (G)$. Moreover the union of the direct system of all $J_p$ is dense in $C^*_{\cpt}(G)$, and as a result 
\begin{equation}
    \label{eq-spectrum-of-compact-ideal-as-increasing-union}
\widehat{C_{\cpt}^*(G)} 
= \bigcup_{p}  \widehat{J_p}.
\end{equation}
So it suffices to prove that each irreducible representation of each $J_p$ is a square-integrable representation of $G$.

It follows from Lemma~\ref{lem:I_S_faithful} that 
$J_p$ is faithfully represented on the Hilbert space $L^2(G){\cdot} p$, where, according to the definition of the compact ideal, it acts as compact operators:
\begin{equation}
    \label{eq-j-p-is-c-zero-direct-sum}
J_p \stackrel{\text{1-1}}\longrightarrow \mathfrak{K}(L^2(G){\cdot} p).
\end{equation}
This means that $J_p$ is abstractly isomorphic to a $c_0$-direct sum of compact operator algebras, and that all of its irreducible representations occur as subrepresentations of the action of $J_p$ on $L^2(G){\cdot}p$; see \cite[Thm.\,1.4.4]{ArvesonInvitation}.  In particular, every irreducible representation of $J_p$ occurs as a subrepresentation of $L^2(G)$, and so, viewed as a representation of $G$, it is a square-integrable representation.
\end{proof}

\begin{lemma}
\label{lem-point-is-closed}
    Let $\pi$ be an irreducible representation of a $C^*$-algebra $A$.  If $\pi(A) = \mathfrak{K}(H_\pi)$, then the singleton $\{ \pi\}$ is a closed subset of the spectrum of $A$.
\end{lemma}

\begin{proof}
This is well known; since $A/\operatorname{ker}(\pi)\cong \mathfrak{K}(H_\pi)$, and since $\mathfrak{K}(H_\pi)$ has a unique irreducible representation up to equivalence, the singleton $\{ \pi\}$ is equal to $\{\tau\in \widehat A : \tau(\operatorname{ker}(\pi))=0\}$ which is by definition a closed set in the spectrum.
\end{proof}

\begin{theorem} 
\label{thm-second-part-of-spectrum-of-compact-ideal-theorem}
Every irreducible unitary representation in the spectrum of the compact ideal is a $K$-admissible topological discrete series representation.
\end{theorem}

\begin{proof}
We shall use the ideals $J_p$ from \eqref{eq-def-of-j-p} in the proof of Lemma~\ref{lem-spectrum-of-compact-ideal-is-only-sq-int-reps}, along with the formula \eqref{eq-spectrum-of-compact-ideal-as-increasing-union}. Because of the latter, it suffices to prove that each irreducible representation   of each $J_p$ is a $K$-admissible topological discrete series representation.

In the proof of Lemma~\ref{lem-spectrum-of-compact-ideal-is-only-sq-int-reps} we showed  that $J_p$ is a $c_0$-direct sum of compact operator algebras, and as a result, its spectrum carries the discrete topology, in which every singleton set  $\{ \pi\}$ is open. Since $J_p$ is an ideal in $C^*_r(G)$, its spectrum is an open subset of the spectrum of $C^*_r(G)$ ---- which is the reduced dual of $G$ --- and so $\{ \pi\}$ is an open subset of the reduced dual, too.

Let us show next  that every irreducible representation $\pi$ of the compact ideal is $K$-admissible. According to Lemma~\ref{lem-spectrum-of-compact-ideal-is-only-sq-int-reps}, $\pi$, viewed as an irreducible unitary  representation of $G$, is square-integrable, and as we noted in \eqref{eq-embedding-sq-int-rep-into-regular-rep} it follows that $L^2(G)$ includes the subrepresentations 
\[
    H_w = \{\, \varphi_{ w,v} \in L^2(G) : v\in H_\pi\,\}
\]
with $w\in H_\pi$ and $\varphi_{ w,v}$ the matrix coefficient function in \eqref{eq-matrix-coefficient-function}, all of which are equivalent to $\pi$.  They generate a closed subspace of $L^2(G)$ that is isomorphic to $H_\pi \otimes \overline H_\pi$ as a left and right representation of $G$.  We obtain, therefore  inclusions 
\begin{equation}
    \label{eq-inclusion-from-square-integrable-matrix-coeffs}
H_\pi \otimes [\overline H_\pi \otimes V]^K \longrightarrow [L^2(G)\otimes V]^K
\end{equation}
of representations of $G$ for all finite-dimensional unitary representations $V$ of $K$.  Now, there is some $f$ in the compact ideal with $\pi(f)\ne 0$.  Because $f$ acts as a compact operator on the right-hand side of \eqref{eq-inclusion-from-square-integrable-matrix-coeffs}, it must act as a compact operator on the left-hand side, too, which implies the finite-dimensionality of $[\overline H_\pi \otimes V]^K$ and so the admissibility of $\pi$.

Finally,  $K$-admissibility implies that $\pi (C^*_r(G))=\mathfrak{K}(H_\pi)$, as explained in the proof of Lemma~\ref{lem-admissible-tds-implies-square-integrable}, so by Lemma~\ref{lem-point-is-closed} the singleton $\{ \pi\}$ is closed.
\end{proof}

\subsection{The compact ideal and real reductive groups}
\label{sec:noncompact_center}
If $G$ is a real reductive group with compact center, then the result about the compact ideal that we shall require is precisely Theorem~\ref{thm:Spec_J} above (in the context of real reductive groups we shall always choose  $K$ to be  a maximal compact subgroup). If $G$ has noncompact center, then we shall need to modify Theorem~\ref{thm:Spec_J}, as described in this section.

If $G$ is any real reductive group, then there is a direct product decomposition 
\begin{equation}
    \label{eq-split-decomposition}
G = \circG \times A_{\Sigma}
\end{equation}
called the \emph{split decomposition} of $G$, in which:
\begin{enumerate}[\rm (i)] 

\item $A_{\Sigma}$ consists of all elements in the center of $G$ that are inverted by the Cartan involution (the notation is taken from \cite[Sec.\,3]{BraddHigsonYunckenOshimaPart1} and will be reviewed here in Section~\ref{sec-parabolic-subgroups}), and

\item $\circG$ is generated, as a closed subgroup of $G$, by all of the compact subgroups of $G$.

\end{enumerate}
See for instance  \cite[Prop.\,7.22(f)]{KnappBeyond}, and see Section~\ref{sec-parabolic-subgroups} for more information about our  conventions regarding real reductive groups.  The group $\circG$ has compact center, and if $G$ has compact center then $\circG = G$.

From \eqref{eq-split-decomposition} we obtain an isomorphism 
\[
C^*_r (G) \stackrel \cong \longrightarrow C^*_r (\circG)\otimes C^*_r (A_{\Sigma}),
\]
using which we may define  a  \emph{compact modulo center} ideal   in $C^*_r(G)$ by
\begin{equation}
    \label{eq-def-of-cmc-ideal}
C^*_{\cmc} (G) \triangleleft C^*_r (G),
\qquad C^*_{\cmc} (G)\stackrel \cong \to C^*_{\cpt}(\circG)\otimes C^*_r (A_{\Sigma}) .
\end{equation}
Theorem~\ref{thm:Spec_J} implies that:
\begin{theorem}
    \label{thm-spectrum-of-cmc-ideal}
    Under the split decomposition in \eqref{eq-split-decomposition} spectrum of the closed, two-sided ideal $C^*_{\cmc} (G) \triangleleft C^*_r (G)$ consists precisely of the tensor products of admissible topological discrete series representations of $\circG$ with unitary characters of the abelian group $A_{\Sigma}$. \qed
\end{theorem}

\begin{definition}
\label{def-tds-mod-center}
    We shall say that an irreducible tempered representation $\pi$ is an \emph{admissible topological discrete series representation, modulo center} if it has the form given in the statement of  Theorem~\ref{thm-spectrum-of-cmc-ideal} above.
\end{definition}
The group $A_{\Sigma}$ acts freely and properly on $G/K$ (we shall think of this as a right action, although of course $A_{\Sigma}$ is central, so it is also the left action).  It therefore acts diagonally on $G/K\times G/K$ and we shall write the quotient manifold as 
\[
\Groupoid_\Sigma = G/K\times_{A_{\Sigma}} G/K ;
\]
the notation is related to the structure of the Satake compactification, as will be reviewed later. If we view  functions $k \in \Cc  (\Groupoid_\Sigma)$ as \continuous\ functions on $G/K\times G/K$ that are   $A_{\Sigma}$-invariant for the diagonal action, then $\Cc  (\Groupoid_\Sigma)$ generates a $C^*$-algebra of operators 
\[
C^*_r(\Groupoid_\Sigma) \subseteq \mathfrak{B} (L^2(G/K))
\]
using the usual formula for integral operators on $L^2 (G/K)$; it is of course a grou\-poid $C^*$-algebra.

More generally, if $V$ is a finite-dimensional unitary representation of $K$, then we may  form the Hilbert space $[L^2(G)\otimes V]^K$, which is the Hilbert space of $L^2$-sections of the vector bundle over $G/K$ induced from $V$, and we may also form the space 
\[
\Cc (\Groupoid_\Sigma; V) = \bigl [ \Cc  ( G\times_{A_{\Sigma}} G) \otimes \End (V)\bigr]^{K\times K} ,
\]
where the first copy of $K$ acts on the right of the first copy of $G$ and on the left of $\End(V)$; and the second copy of $K$ acts on the right of the second copy of $G$ and on the right of $\End(V)$.  This identifies with the space of \continuous\  compactly supported sections of the  bundle over $\Groupoid_\Sigma$ that is obtained by pulling back the vector bundles associated to $V$ and $V^*$ along the two projection maps $\Groupoid_\Sigma \rightrightarrows G/K$, and then forming the tensor product, as we shall discuss in more detail later. But in any case we may form the completion 
\[
C^*_r(\Groupoid_\Sigma;V) \subseteq \mathfrak{B}  ( L^2(G;V)  ).
\]

\begin{lemma} 
\label{lem-groupoid-c-star-algebra-as-compacts}
The unitary isomorphism
\[
L^2(G;V)  \stackrel \cong \longrightarrow L^2(\circG/K;V) \otimes L^2 (A_{\Sigma})
\]
given by the split decomposition \eqref{eq-split-decomposition} conjugates the $C^*$-algebra $C^*_r (\Groupoid_\Sigma; V)$ to the $C^*$-algebra
\[
\mathfrak{K} \bigl ( L^2(\circG/K;V) \bigr ) \otimes C^*_r (A_{\Sigma})\subseteq \mathfrak{B}\bigl (L^2(\circG/K;V) \otimes L^2 (A_{\Sigma}) \bigr )
\]
\end{lemma} 

\begin{proof} 
This follows from the right $(K{\times}K)$-equivariant diffeomorphism
\[
\begin{gathered}
G\times_{A_{\Sigma}} G \stackrel \cong \longrightarrow 
(\circG\times \circG) \times A_{\Sigma}
\\
(\circg_2 a_2, \circg_1a_1) \longmapsto \bigl ( (\circg_2,\circg_1), a_2a_1^{-1}\bigr )
\end{gathered}
\]
and the usual characterization of the compact operators as the norm closure of the compactly supported smoothing operators.
\end{proof}

The following result offers a characterization of the compact modulo center ideal, along the lines of our definition of the compact ideal, that is more conceptual than the \emph{ad hoc} definition we gave in \eqref{eq-def-of-cmc-ideal}.

\begin{theorem}
    \label{theorem-comapct-mod-center-ideal}
 The    ideal $C^*_{\cmc} (G) \triangleleft C^*_r (G)$ is precisely the ideal of all  those elements in $C^*_r(G)$ that map into the $C^*$-algebras 
 \[
 C^*_r (\Groupoid_\Sigma ; V)\subseteq \mathfrak{B}( L^2(G/K;V) )
 \]
 under every regular representation 
 \[
 C^*_r (G) \longrightarrow  \mathfrak{B}  ( L^2(G/K;V) ),
 \]
 for every finite-dimensional unitary representation $V$ of $K$.
\end{theorem}

\begin{remark}
Lemma~\ref{lem-groupoid-c-star-algebra-as-compacts} shows that the $C^*$-subalgebra in the statement of the theorem is indeed an ideal.
\end{remark}

\begin{proof}[Proof of the Theorem] 
This follows from the commuting diagram
\[
\xymatrix{
C^*_{\cmc} (G) \ar[r]\ar[d]_{\cong} & \mathfrak{B}\bigl (L^2(G/K;V)\bigr)\ar[d]^{\cong}
\\
C^*_{\cpt} (\circG)\otimes C^*_r (A_{\Sigma}) \ar[r] &
\mathfrak{B}\bigl (L^2(\circG/K;V)\otimes L^2 (A_{\Sigma})\bigr )
}
\]
and the definition of $C^*_{\cpt} (\circG)$.
\end{proof}

\section{Parabolic induction}
\label{sec-parabolic-induction}

The purpose of this section is to record a $C^*$-algebraic characterization (in terms of ideals in the reduced group $C^*$-algebra) of the tempered irreducible representations of a real reductive group $G$  that may be embedded in parabolically induced representations.  The material below is quite well known; for instance it is implicit in \cite{ClaCriHig:parabolic_induction}.

\subsection{Parabolic subgroups} 
\label{sec-parabolic-subgroups}
We shall use the definition of  real reductive group  from  \cite[Sec.\,VII.2]{KnappBeyond}, which includes all the groups in  Harish-Chandra's class \cite[Sec.\,3]{HarishChandra75}; the two definitions  coincide   for linear   groups  \cite[pp.\,447-449]{KnappBeyond}. 

Choose an Iwasawa decomposition $G {=} KAN$ \cite[Prop.\,7.31]{KnappBeyond}.   The \emph{standard parabolic subgroups} of $G$ are the parabolic subgroups of $G$ \cite[Sec.\,VII.7]{KnappBeyond} that include the \emph{minimal parabolic subgroup} $P_{\min}{=}MAN$, where $M$ is the centralizer of $A$ in $K$. Our choice of  Iwasawa decomposition determines a system of positive restricted roots $\Delta^+(\mathfrak{g},\mathfrak{a})$, and within this a system of simple roots $\Sigma$. Associated to each subset $I\subseteq \Sigma$ is a standard parabolic subgroup
\begin{equation}
    \label{eq-standard-parabolic-subgroup}
P_I = M_I A_I N_I \subseteq G.
\end{equation}
The groups $P_I$ are all of the standard parabolic subgroups, and every other parabolic subgroup is conjugate to some $P_I$. For all this, see e.g.\ \cite[Sec.\,VII.7]{KnappBeyond}, or  see \cite[Sec.\,3.1]{BraddHigsonYunckenOshimaPart1} for a  summary adapted to the perspectives of this paper.

\begin{example}
    If $G=SL(n,\R)$, and if $P_{\min}$ is the subgroup of upper-trian\-gular matrices in $G$, then the standard parabolic subgroups are the various block upper-triangular subgroups of $G$, including both $P_\emptyset=P_{\min}$ and $P_\Sigma=G$ itself.
\end{example} 

\subsection{Parabolic induction via Hilbert C*-modules}
For the rest of Section~\ref{sec-parabolic-induction} we shall fix a real reductive group $G=KAN$ and a standard parabolic subgroup $P_I=M_IA_IN_I$.

The Cartesian product group $L_I = M_I A_I$ is the associated \emph{Levi factor}, and if $\sigma$ is a unitary representation of $L_I$, then we may form the unitarily induced representation 
$ \Ind_{P_I}^G \sigma $ \cite[Ch.\,VII]{Knapp:representation_theory} by first extending $\sigma$ to a representation of $P_I$ that is trivial on the normal subgroup $N_I$, and then applying the functor of unitary induction.  We shall always assume that $\sigma$  is weakly contained in the regular representation of $L_I$, in which case $ \Ind_{P_I}^G \sigma $ is weakly contained in the regular representation of $G$.   

\begin{remark}
    \label{rem-weak-containment} 
We shall use the notion of weak containment of representations,  see \cite[Def.\,3.4.5]{DixmierCStarBook},  extensively below. For most purposes (such as the one above) the definition of weak containment  is conveniently   formulated in terms of matrix coefficient functions. But it will be important to us that at the $C^*$-algebra level there is an equivalent formulation in terms of ideals:  $\pi$ is weakly contained in $ \rho$ if and only if $\operatorname{ker}(\rho)\subseteq \operatorname{ker}(\pi)$; see  \cite[Thm.\,3.4.3]{DixmierCStarBook}.
\end{remark}

The functor of parabolic induction has been recast in terms of Hilbert $C^*$-modules by Pierre Clare \cite{Clare:parabolic_induction}, and his perspective will be very convenient for us. What follows is a summary of the  variation on Clare's definitions that is appropriate for tempered representations; see \cite[Sec.\,4]{ClaCriHig:parabolic_induction} for more details. 

The  Banach space $C^*_r(G/N_I)$ is a certain norm completion of the space of \continuous, compactly supported functions on $G/N_I$.  It carries a   Hilbert $C^*_r (L_I)$-module structure associated to the free and proper right-translation action of $L_I$ on $G/N_I$, and a  unitary action of $G$ that integrates to a left action of $C^*_r(G)$ by adjointable operators on $C^*_r(G/N_I)$. Clare's main observation is as follows (see \cite[Ch.\,4]{LanceHilbertModules} for the tensor product involved in the statement, and for other information about Hilbert modules): 

\begin{proposition}[See {\cite[Cor.\,1]{Clare:parabolic_induction}]}]
\label{lem-clare-lemma}
The functor of parabolic induction \textup{(}from unitary representations of $L_I$ that are weakly contained in the regular representation to unitary representations of $G$ that are weakly contained in the regular representation\textup{)} is equivalent to the tensor product functor $H_\sigma \mapsto C_r^*(G/N_I) \otimes_{C_r^*(L_I)} H_\sigma$.
\end{proposition}

\subsection{A C*-algebra characterization of parabolically induced representations}
We shall use Clare's perspective on parabolic induction to prove the following result: 

\begin{theorem}
\label{thm:weakly_in_quasireg_is_induced}
  The following two classes of  irreducible tempered representations   $G$ are equal to one another: 
    \begin{enumerate}[\rm (i)]
    
    \item The irreducible tempered representations of $G$ that may be embedded as subrepresentations of some parabolically induced representation $\Ind_{P_I}^G \sigma$, where $\sigma$ is some irreducible tempered representation of  $L_I$.
    
    \item The irreducible tempered representations of $G$ that are weakly contained in the regular representation of $G$ on $L^2 (G/N_I)$.
    \end{enumerate}
\end{theorem}

Denote by $\mathfrak{K}(C_r^*(G/N_I))$ the $C^*$-algebra of compact \textup{(}in the Hilbert module sense\textup{)} adjointable Hilbert module operators on $C^*_r (G/N_I)$ \cite[p.10]{LanceHilbertModules}.

\begin{lemma}[See {\cite[Prop.\,4.5]{ClaCriHig:parabolic_induction}}] 
\label{lem-compact-action-on-clare-module}
The $C^*$-algebra $C^*_r(G)$ acts on the Hilbert $C^*$-module $C^*_r (G/N_I)$ through the $C^*$-algebra $\mathfrak{K}(C_r^*(G/N_I))$ of compact Hilbert $C^*$-module operators on $C_r^*(G/N_I)$.
\end{lemma} 

\begin{lemma}
\label{lem-irreps-of-comapct-ops}
 Up to unitary equivalence, the 
irreducible representations of $\mathfrak{K}(C_r^*(G/N_I))$ are precisely those of the form
\[
\begin{aligned}
\mathfrak{K}(C_r^*(G/N_I)) 
    & \longrightarrow \mathfrak{K}\bigl (
C_r^*(G/N_I)\otimes _{C^*_r(L_I)} H_\sigma \bigr )
\\
T   
    & \longmapsto T\otimes I ,
\end{aligned}
\] 
where $\sigma$ is a tempered irreducible unitary representation of $L_I$.
\end{lemma}

\begin{proof}  
The Hilbert $C^*$-module 
$C_r^*(G/N_I)$ is a full Hilbert $C_r^*(L_I)$-module (meaning that the closed two-sided ideal generated by all inner products is $C^*_r(L_I)$ itself). Therefore, considered as a  left $\mathfrak{K}(C_r^*(G/N_I))$-module,  the Hilbert $C^*_r(L_I)$-module 
$C_r^*(G/N_I)$ is   a strong Morita equivalence between  $\mathfrak{K}(C_r^*(G/N_I))$ and $C_r^*(L_I)$; see \cite[Ch.\,7]{LanceHilbertModules}. The lemma follows from this.
\end{proof}

\begin{lemma}
\label{lem:l2_and_cstar_weakly_equivalent}
The regular representations of $C^*_r(G)$ on $C^*_r(G/N_I)$  and  on $L^2 (G/N_I)$  have the same kernels: 
\[
 \ker\bigl(C_r^*(G) \to \mathfrak{K}(C_r^*(G/N_I))\bigr)
 =\ker\bigl(C_r^*(G) \to \mathfrak{B}(L^2(G/N_I))\bigr) .
\]
\end{lemma}

\begin{proof} 
On \continuous, compactly supported functions $\varphi,\psi\colon G/N\to \C$, the  $C^*_r(L_I)$-valued inner product on  $C^*_r(G/N_I)$ is given  by the formula
\begin{equation*}
\langle \varphi, \psi\rangle _{C^*_r (G/N_I)}(\ell) = \langle \varphi, \psi\cdot \ell^{-1} \rangle _{L^2(G/N_I)} ,
\end{equation*}
where the dot indicates the right unitary action of $L_I$ on $L^2 (G/N_I)$ \cite[p.\,10]{ClaCriHig:parabolic_induction}.  If we write this as an identity of \continuous, compactly supported functions on $L_I$,
\begin{equation*}
\langle \varphi, \psi\rangle _{C^*_r (G/N_I)}  = \langle \varphi, \pmb{\psi}   \rangle _{L^2(G/N_I)} ,
\end{equation*}
where $\pmb \psi$ denotes the function $\ell\mapsto \psi\cdot \ell^{-1}$, and if $\theta$ is any \continuous, compactly supported function on $L_I$, then we obtain 
\begin{equation*}
\langle \varphi, \psi\rangle _{C^*_r (G/N_I)}* \theta   = \langle \varphi, \pmb{\psi}   \rangle _{L^2(G/N_I)} * \theta.
\end{equation*}
From this we obtain 
\begin{equation}
    \label{eq-def-clare-module-inner-product-2}
\langle \lambda_{C^*_r(G/N_I)}(f)\varphi, \psi\rangle _{C^*_r (G/N_I)}* \theta   = \langle\lambda_{L^2(G/N_I)}(f) \varphi, \pmb{\psi}   \rangle _{L^2(G/N_I)} * \theta,
\end{equation}
first for $f\in \Cc  (G)$, but then by continuity for all $f\in C^*_r(G)$, if we view \eqref{eq-def-clare-module-inner-product-2} as an identity of continuous functions on $G$ (not necessarily compactly supported).  The lemma follows from this.
\end{proof}

\begin{proof}[Proof of Theorem \ref{thm:weakly_in_quasireg_is_induced}]
The inclusion of class (i) in the theorem into class (ii) is clear from Clare's description of parabolic induction and Lemma~\ref{lem:l2_and_cstar_weakly_equivalent}.

Suppose that $\pi$ is weakly contained in $L^2(G/N_I)$.  By Lemma \ref{lem:l2_and_cstar_weakly_equivalent}, $\pi$ is trivial on the kernel of the representation 
\[
C_r^*(G) \longrightarrow  \mathfrak{K}(C_r^*(G/N_I)) ,
\]
and so it can be viewed as a representation of the image, which is a $C^*$-subalgebra of $\mathfrak{K}(C_r^*(G/N_I))$.  But now every irreducible representation of a subalgebra $A$ of a $C^*$-algebra $B$ can be embedded into the restriction to $A$ of an irreducible representation of $B$; see \cite[Prop.\,2.10.2]{DixmierCStarBook}. It therefore follows from Lemma~\ref{lem-irreps-of-comapct-ops} that $\pi$ may be embedded into some $\Ind_{P_I}^G \sigma$, as required.
\end{proof}

\subsection{The cuspidal ideal}
\label{sec-cuspidal}

Let us borrow a term from Harish-Chandra (see e.g.\ \cite[Sec.\,18]{HarishChandra75}) and make the following definition: 

\begin{definition}
\label{def-cuspidal-rep}
    The \emph{cuspidal ideal} $C^*_{\cusp}(G) \triangleleft  C^*_r (G)$ is the intersection of ideals  
    \[
     \operatorname{ker} \bigl ( \lambda_{G/N_I}\colon C^*_r (G) \to \mathfrak{B} ( L ^2 (G/ N_I))\bigr )
    \]
    as $I$ ranges over all subsets of $\Sigma$ other than $\Sigma$ itself.
\end{definition}

\begin{theorem} 
\label{thm-cuspidal-reps-nonzero-on-cuspidal-ideal}
A tempered irreducible unitary representation of $G$  vanishes on cuspidal ideal, and so determines an irreducible representation  of the quotient $C^*$-algebra $C^*_r(G) / C^*_{\cusp}(G)$, if and only if   it embeds in a unitary representation of the form
    $\Ind _{P_I}^G \sigma$,  where $I\subseteq \Sigma$ is a proper subset of the set of simple restricted roots, and  $\sigma$ is a tempered  irreducible unitary representation of $L_I$.
\end{theorem}

\begin{proof} 
Let $\pi$ be a tempered irreducible unitary representation of $G$, viewed as an irreducible representation of $C^*_r(G)$.  Suppose first that  $\pi$ is \emph{nonzero} on the cuspidal ideal. Then of course $\pi$ is nonzero on each individual $\operatorname{ker}(\lambda_{G/N_I})$, and so by Theorem~\ref{thm:weakly_in_quasireg_is_induced}, it  can embed in no parabolically induced representation, as in the statement of the theorem.

Suppose next that $\pi$ is \emph{zero} on the cuspidal ideal. 
It is a general fact about $C^*$-algebras that if $\rho$ is any irreducible representation of a $C^*$-algebra $A$,  if $J_1,\dots J_k$ are closed, two-sided ideals in $A$, and if $\rho$ is zero on $J_1\cap\cdots \cap J_k$, then $\rho$ is zero on at least one of $J_1,\dots ,J_k$; see \cite[Lem.\,2.11.4]{DixmierCStarBook}.  So  $\pi$ vanishes on at least one of the ideals  $\operatorname{ker}  ( \lambda_{G/N_I})$ in Definition~\ref{def-cuspidal-rep}, and therefore by Theorem~\ref{thm:weakly_in_quasireg_is_induced} it embeds in some  $\Ind _{P_I}^G \sigma$, as required.
\end{proof}

\section{The Satake groupoid and its groupoid C*-algebra} 
\label{sec-satake-groupoid-and-c-star-algebra}
In this section we shall introduce and study the $C^*$-algebra of the Satake groupoid.  The groupoid was constructed in \cite{BraddHigsonYunckenOshimaPart1}, and we shall begin, in Section~\ref{sec-satake-space-and-groupoid}, with a review of some of the topics there.   After that and some further preliminaries, we shall construct a morphism from the reduced $C^*$-algebra of $G$ to the groupoid $C^*$-algebra that is central to our approach to the Harish-Chandra principle. 

\subsection{The Satake compactification and the Satake groupoid}
\label{sec-satake-space-and-groupoid}

Let $G=KAN$ be a real reductive group. Any two maximal compact subgroups of $G$ are conjugate, and it follows from the $KAK$-decomposition \cite[Thm.\,7.39]{KnappBeyond} that the normalizer of the distinguished maximal compact subgroup $K$ is $K{\cdot} A_{\Sigma}$. So the set of all maximal compact subgroups of $G$ may be identified with $G/K{\cdot} A_{\Sigma}$, which gives it a topology, and indeed a smooth $G$-manifold structure. 

The (maximal) \emph{Satake compactification} for $G$ is a $G$-compactification of the space of all maximal compact subgroups.  See \cite{Satake60} for the original construction, and \cite{BorelJiCompactificationsBook,GuivarchJiTaylor}, as well as \cite{BraddHigsonYunckenOshimaPart1}, for other accounts.  

The Satake compactification $\Satake$  may be constructed in various ways, but for our purposes it is best approached via the method of Oshima \cite{Oshima78}, who embedded $\Satake$ into a smooth closed $G$-manifold $\Oshima$.  This \emph{Oshima space} $\Oshima$ includes a finite family of smooth, closed hypersurfaces, each invariant under $G$, that intersect one another in normal crossings, and the space $G/K{\cdot} A_{\Sigma}$  is embedded in $\Oshima$ as a single component in the mutual complement of these hypersurfaces.  This gives $\Satake$ the structure of a smooth compact $G$-manifold with corners.  See \cite{BraddHigsonYunckenOshimaPart1} for a  review tailored to our purposes.

The Satake compactification has finitely many $G$-orbits.  They are in bijection with the subsets of the set $\Sigma$ of simple restricted roots (including $\Sigma$ itself), as follows: for each $I\subseteq \Sigma$ there is a unique point $x_I\in \Satake$ whose isotropy group is 
\begin{equation}
    \label{eq-isotropy-of-x-i}
    G_{x_I} = K_I A_I N_I ,
\end{equation}
where $A_I$ and $N_I$ are the factors in the standard parabolic subgroup $P_I{=}M_IA_IN_I$ in \eqref{eq-standard-parabolic-subgroup} and $K_I = K \cap M_I$.  The set $\{\, x_I : I \subseteq \Sigma\,\}$ is a complete set of representatives of the $G$-orbits.

\begin{definition}
\label{def-x-i-orbit}
    For $I\subseteq \Sigma$ we shall denote by $\Satake_I\subseteq \Satake$ the orbit of the point $x_I$ above.
\end{definition}

\begin{example}
The orbit $\Satake_{\Sigma}\subseteq \Satake$ is the interior   of the Satake compactification ($K_{\Sigma}=K$ and $  A_\Sigma = A_{\Sigma} $).
\end{example}

In \cite{BraddHigsonYunckenOshimaPart1} we constructed a Lie groupoid $\Groupoid_{\Oshima}$ whose object space is $\Oshima$, following a   method that was introduced by  Omar Mohsen in \cite{Mohsen:blowup} in the context of foliations.  The construction is carried out in two stages.  First a family of closed subgroups $H_m\subseteq G$, parametrized by the points $m\in\Oshima$, is constructed in such a way that 
\begin{enumerate}[\rm (i)]

\item $H_m$ is a subgroup of the isotropy group of $m\in \Oshima$,

\item if $m\in \Oshima$ and $g\in G$ then $H_{g{\cdot}m} = \Ad_g (H_m)$, and 

\item the total space of the family,
\[
\Subgroupoid_{\Oshima} = \{\, (h,m) : h\in H_m,\,\,\, m\in \Oshima\,\} ,
\]
is a closed subset and a smooth submanifold of the Cartesian product space $G {\times } \Oshima$, for which the projection to $\Oshima$ is a submersion.
\end{enumerate}
See \cite[Sec.\,3.5]{BraddHigsonYunckenOshimaPart1}.  Second, the Oshima groupoid is constructed as a quotient of the transformation groupoid 
\[
G \ltimes \Oshima = \{ \, (m_2,g,m_1) : m_1,m_2\in \Oshima,\,\,\, g\in G,\,\,\, gm_1 = m_2 \,\}
\]
 by $\Subgroupoid_{\Oshima}$, upon viewing $\Subgroupoid_{\Oshima}$  as the normal subgroupoid consisting of all those morphisms $(m,h,m)$ for which  $h\in H_m$.    This is a Lie groupoid \cite[Def.\,4.2.2]{BraddHigsonYunckenOshimaPart1}.

The Satake compactification $\Satake\subseteq \Oshima$, being a $G$-invariant subset of $\Oshima$, is a saturated subset for both the transformation and Oshima groupoids, and we define the \emph{Satake groupoid} $\Groupoid_{\Satake}$ to be the reduction of $\Groupoid_{\Oshima}$ to $\Satake$, or in other words the subgroupoid comprised of morphisms with sources and targets in $\Satake$. This is not a Lie groupoid, because the object space is not a smooth manifold (without boundary or corners).  But it is smooth in the source and target fiber directions (it is a \emph{continuous family groupoid} in the terminology of \cite{Paterson00}), and this suffices for our purposes.  Compare \cite[Remark\,2.6]{DebordSkandalisSurvey19}.

\subsection{Haar system on the Satake groupoid}
\label{sec-haar-system}
Because it is  the reduction of the  Lie groupoid $\Groupoid_{\Oshima}$ to a closed saturated subset, the Satake groupoid is guaranteed to possess a Haar system, which may be constructed using densities.  

As we mentioned, our main objective in Section~\ref{sec-satake-groupoid-and-c-star-algebra} is the construction of a morphism from the $C^*$-algebra of $G$ to the  $C^*$-algebra of the Satake groupoid $\Groupoid_{\Satake}$. In doing so we shall be  closely following ideas of  Omar Mohsen \cite{Mohsen:blowup}, but whereas Mohsen uses the density approach to Haar systems, here we shall use explicit measures,  constructed as follows.

Fix a smooth family of Haar measures   on the subgroups $H_m\subseteq G$ that comprise the family $\Subgroupoid_{\Oshima}$  over the Oshima space $\Oshima$ (these  are unimodular Lie groups, incidentally).  

\begin{lemma} 
\label{lem-introducing-the-delta-function}
There is a unique smooth function
\[
\delta \colon G \times \Oshima \longrightarrow (0,\infty)
\]
such that  if $g\in  G $ and $m\in  \Oshima$, and if $f\colon H_m\to \C$ is any \continuous, compactly supported function, then 
    \[
    \int _{H_m} f(h) \, d h = \delta(g,m) \int_{H_{g{\cdot}m}} f(g^{-1}hg) \, d h.
    \]
\end{lemma}

\begin{proof}
    This is a consequence of usual uniqueness theorem for Haar integrals, since the right-hand side above may be viewed as a Haar integral on $H_m$.
\end{proof}

Now fix a Haar measure on $G$.  With this, and with the chosen Haar measures on the groups $H_m$, there are unique   right-$G$-invariant smooth measures   on the homogeneous spaces 
$ H_m \backslash G$ for which  
\begin{equation}
\label{eq-def-of-mu-m}
\int_G f(g) \, dg = \int _{H_m \backslash G}\left (
\int_{H_m} f(hg) \, d h  \right ) d\dot g
\end{equation}
for all \continuous, compactly supported $f\colon G \to \C$.  The dot in $\dot g$ is a reminder that we are integrating over cosets.  

Now recall that a Haar system on $\Groupoid_{\Satake}$ is a smooth family of smooth measures on the target fibers 
\[
\Groupoid_{\Satake}^x = \{ \, \gamma \in \Groupoid_{\Satake} : \operatorname{target}(\gamma) = x\,\} ,
\]
such that
\begin{equation}
    \label{eq-def-haar-system}
\int_{\Groupoid_{\Satake}^{x_1}} f(\alpha \circ \beta ) \, d\mu^{x_1} (\beta ) =
\int_{\Groupoid_{\Satake}^{x_2}} f(\gamma  ) \, d\mu^{x_2} (\gamma )
\end{equation}
for all $\alpha \in \Groupoid_{\Satake}$ with 
\[
\operatorname{source}(\alpha) = x_1\quad \text{and} \quad 
\operatorname{target}(\alpha) = x_2,
\]
and all $f\in \Cc  (\Groupoid_{\Satake})$.   See \cite[Sec.\,I.2]{RenaultGroupoidApproach80} or \cite[Sec.\,2.2]{DebordSkandalisSurvey19}.
A straightforward computation now shows that:

\begin{lemma} The family of integrals 
\[
\int_{\Groupoid_{\Satake}^x} f \, d\mu^x  
=  \int _{H_x\backslash G} f(x, g,g^{-1}x) \delta(g^{-1},x) \, d\dot g
\qquad (x\in \Satake,\,\,\, f\in \Cc  (\Groupoid_{\Satake}))
\]
is a Haar system on the Satake groupoid. \qed
\end{lemma}

\begin{remark}
    Here and below we are representing elements of the Satake grou\-poid as equivalence classes of triples $(x_2,g,x_1)$ with $gx_1=x_2$.  The
    the equivalence relation is
\[
( x_2, g , x_1   ) \sim (x_2,gh,x_1)\qquad \forall h\in H_{x_1}
\]
and the groupoid structures are 
      \[
\begin{aligned}
\operatorname{source} ( x_2, g , x_1   )  
    & = x_1
\\
\operatorname{target}   ( x_2,g,x_1   )
    & = x_2
\\
  (x_3, g_2 , x_2 ) \circ    ( x_2, g_1 , x_1 ) & =  ( x_3, g_2 g_1, x_1 ) .  
\end{aligned}
\]
\end{remark}

\subsection{Regular representations of the groupoid algebra}
\label{sec-regular-representations}
Using Haar system, we may define a convolution multiplication and $*$-operation on $\Cc  (\Groupoid_{\Satake})$ in the usual way, see   \cite[Sec.\,II.1]{RenaultGroupoidApproach80}  or \cite[Sec.\,2.2]{DebordSkandalisSurvey19}, and so equip $\Cc  (\Groupoid_{\Satake})$ with the structure of a $*$-algebra.  This may be completed to obtain the \emph{reduced $C^*$-algebra} $C^*_r(\Groupoid_{\Satake})$ using the regular representations of the groupoid, which we briefly review.

For $x\in \Satake$, denote by $\Groupoid_{\Satake,x}$ the source fiber 
\[
\Groupoid_{\Satake,x} = \{\, \gamma \in \Groupoid_{\Satake} : \operatorname{source}(\gamma) = x\,\}.
\]
The Hilbert space $L^2 (\Groupoid_{\Satake,x})$ is the completion of $\Cc  (\Groupoid_{\Satake,x})$ in the norm given by the formula 
\[
\| \xi \|^2 _{L^2 (\Groupoid_{\Satake,x})} = \int_{\Groupoid_{\Satake}^x} | \xi (\gamma^{-1})|^2 \, d\mu^x (\gamma),
\]
and $\Cc  (\Groupoid_{\Satake})$ is represented as bounded operators on this Hilbert space via the \emph{regular representation} 
\begin{equation}
    \label{eq-groupoid-regular-rep}
(\lambda_x(f)\xi )(\gamma) =  \int _{\Groupoid_{\Satake}^{\operatorname{target}(\gamma)}} f(\alpha)\xi (\alpha^{-1} \circ \gamma ) \, d\mu^{\operatorname{target}(\gamma)}(\alpha).
\end{equation}
See \cite[Sec.\,2.3]{DebordSkandalisSurvey19}.   The reduced $C^*$-algebra $C^*_r(\Groupoid_\Satake)$ is then the completion of $\Cc (\Groupoid_\Satake)$ in the norm
\begin{equation}
    \label{eq-reduced-groupoid-c-star-norm}
\| f\|_{C^*_r (\Groupoid_{\Satake})} = \sup_{x\in \Satake} \| \lambda _x (f)\| _{\mathfrak{B}(L^2 (\Groupoid_{\Satake,x}))}.
\end{equation}

The regular representations can be converted into forms more congenial for harmonic analysis as follows.
For $x\in \Satake$, form the Hilbert space $L^2 (G/H_x)$ by completing  $\Cc  (G/H_x)$ in the  norm given by the formula 
\begin{equation}
    \label{eq-norm-in-l-2-g-mod-h}
\| \varphi \|^2 _{L^2 (G/ H_x)} = \int_{H_x\backslash G} | \varphi (g^{-1})|^2 \, d \dot g,
\end{equation}
in which we use the integral on $H_x \backslash G$ from \eqref{eq-def-of-mu-m}.

\begin{definition} 
\label{def-tilde-xi}
For $x\in \Satake$ and  $\xi \in \Cc  (\Groupoid_{\Satake,x})$ define 
$\tilde \xi \in \Cc  (G/H_x)$ by 
\[
\tilde \xi (g) =  \delta (g,x)^{1/2}{\cdot}\xi (gx,g,x)  \qquad (g\in G).
\]
\end{definition}

\begin{lemma} 
    \label{lem-l-2-isometric-isomorphism}
The morphism
\[
\Cc  (\Groupoid_{\Satake,x}) \ni \xi \longmapsto \tilde \xi \in \Cc  (G/H_x)
\]
extends to a unitary isomorphism 
\[
\pushQED{\qed}
L^2 ( \Groupoid_{\Satake,x}) \stackrel \cong \longrightarrow L^2 (G/H_x).
\qedhere \popQED
\]
\end{lemma} 

\subsection{The integration morphism}
\label{sec-integration-morphism}
Let $x\in \Satake$ . The group convolution $*$-algebra  $\Cc(G)$ is represented as a $*$-algebra of bounded operators on   $L^2 (G/H_x)$  by the usual  formula 
\begin{equation}
    \label{eq-group-regular-rep}
(\lambda_{G/H_x}(f)\varphi )(g) = \int_G f(g_1) \varphi(g_1^{-1} g) \, dg_1
\end{equation}
for $f\in \Cc  (G)$ and $\varphi\in \Cc  (G/H_x)$. This representation is linked to the regular representations of $\Cc(\Groupoid_\Satake)$ in the previous section by the follow$C^*_r(\Groupoid_\Satake)$ing construction:

\begin{definition}
    \label{def-integration-morphism}
The \emph{integration morphism} is the  linear map
\begin{equation*}
    \label{eq-integration-morphism-on-test-functions}
\Cc  (G) \ni f \longmapsto f_{!} \in \Cc  (\Groupoid_{\Satake})
\end{equation*}
that is defined by means of the formula 
\begin{equation*}
f_{!}(gx,g,x) 
     = \delta  (g,x)^{-1/2} \int _{H_x} f(gh)\, dh 
\end{equation*}
for all $(gx,g,x)\in \Groupoid_\Satake$.
\end{definition}

\begin{lemma} 
\label{lem-compatibility-of-regular-reps-via-integration}
For every $x \in \Satake$ there  is a commuting diagram 
\[
\xymatrix{
\Cc  (G)\ar[d]_{\lambda_{G/H_x}} \ar[r]& \Cc  (\Groupoid_{\Satake})\ar[d]^{\lambda _x}
\\
\mathfrak{B}(L^2 (G/H_x)) \ar[r]_\cong & \mathfrak{B}(L^2 (\Groupoid_{\Satake,x}))
}
\]
in which the top arrow is the integration morphism, and the bottom arrow is induced from the inverse of the unitary isomorphism in Lemma~\ref{lem-l-2-isometric-isomorphism}. \qed
\end{lemma}

\begin{corollary} 
The integration morphism  in Definition~\textup{\ref{def-integration-morphism}} is a morphism of $*$-algebras from 
$\Cc   (G)$ to $ \Cc (\Groupoid_{\Satake})$.
\end{corollary}

\begin{proof} 
This follows from Lemma~\ref{lem-compatibility-of-regular-reps-via-integration} since the  representations of both $\Cc  (G)$ and $\Cc  (\Groupoid_\Satake)$ are $*$-algebra morphisms, and since the latter are collectively faithful (the intersection of the kernels of all $\lambda_x$  is zero).
\end{proof}

\subsection{Integration morphism at the C*-algebra level}
\label{sec-integration-morphism-at-c-star-level}
We want to promote the integration morphism  to a morphism of reduced $C^*$-algebras. To do so, we need to recall from \cite{BraddHigsonYunckenOshimaPart1} that for any $I\subseteq \Sigma$, the group $H_{x_I}$, which we shall write simply as $H_I$ is 
\begin{equation}
\label{eq-def-of-h-i}
H_{I} = K_I N_I \subseteq G,
\end{equation}
and that every member of the family $\Subgroupoid_{\Satake}$ is conjugate in $G$ to some (unique) $H_I$.  

\begin{theorem} 
\label{thm-integration-morphism-at-c-star-level}
The integration morphism  in Definition~\textup{\ref{def-integration-morphism}} extends to a morphism of $C^*$-algebras
\[
C^*_r (G) \longrightarrow C^*_r(\Groupoid_{\Satake})
\]
For every $x\in \Satake$, the composition of the integration $C^*$-algebra morphism with the regular representation of $C^*_r(\Groupoid_{\Satake})$ on $L^2(\Groupoid_{\Satake,x})$ is unitarily equivalent to the regular representation of $C^*_r (G)$ on $L^2 (G/ H_x)$.
\end{theorem}

\begin{proof} 
Each $H_I$ is an amenable group, and therefore the quasi-regular representation of $G$ on any $L^2 (G/H_I)$ is weakly contained in the regular representation of $G$ \cite[Thm.\,4.1]{Fell61}.  As a result, 
\[
\| f\| _{C^*_r(G)} \ge \| 
\lambda_{G/H_I} (f)\|_{\mathfrak{B}(L^2 (G/H_I))} 
\qquad \forall I\subseteq \Sigma ,
\]
and the same holds for all $\lambda_{G/H_x}$, each of which is unitarily equivalent to one of the representations $\lambda_{G/H_I}$.

The extendibility of the integration morphism is now a consequence of the compatibility of regular representations described in Lemma~\ref{lem-compatibility-of-regular-reps-via-integration}, and the fact that the regular representations $\lambda _x$ define the norm of $C^*_r(\Groupoid_{\Satake})$ as in \eqref{eq-reduced-groupoid-c-star-norm}. The unitary equivalence between representations in the statement of the theorem is provided by Lemma~\ref{lem-compatibility-of-regular-reps-via-integration}.
\end{proof}

\section{Vector bundles on the Satake compactification}
\label{sec-vector-bundles-on-satake}

In this section we shall present a modest elaboration of the groupoid $C^*$-algebra that will be important for our harmonic-analytic applications.

\subsection{Vector bundles}
\label{sec:V}

If $ V$ is a finite-dimensional representation of $K$, then of course associated to $V$ there is an $G$-equivariant Hermitian vector bundle over $G/K{\cdot} A_{\Sigma}$ namely  $G/A_{\Sigma} \times _K V$.  This bundle may be extended to a $G$-equivariant Hermitian vector bundle over the Satake compactification: 

\begin{theorem}
\label{thm:V}
    Let $\pi\colon K\to V$ be a    finite-dimensional unitary representation   of $K$.
    There is a $G$-equivariant Hermitian vector bundle $V_{\Satake}$ on the Satake compactification $\Satake$ with the following property: If $\Satake_I\subseteq \Satake$ is the $G$-orbit in $\Satake$ associated to $I \subseteq \Sigma$, and if $V_{\Satake_I}$ is the restriction of the vector bundle $V_\Satake$ to $\Satake_I$, then there is an isomorphism of $G$-equivariant vector bundles 
    \[
    \xymatrix{
   G/A_IN_I\times_{K_I} V \ar[r]^-{\cong} \ar[d] &   V_{\Satake_I}  \ar[d]
    \\
    G / K_IA_IN_I\ar[r]_-{\cong}  & \Satake_I 
    }
    \]
    in which the bottom map is the orbit map for $x_I\in \Satake_I$.
\end{theorem}

Since the construction of  $V_{\Satake}$ involves an examination of how the Satake compactification  itself is constructed,  we shall defer it to Appendix~\ref{app-vector-bundle}.

\subsection{Fell bundles on the Satake groupoid}

Let $V$ be a finite-dimensional unitary representation $V$ of $K$, and let $V_{\Satake}$ be a  $G$-equi\-variant Hermitian vector bundle over $\Satake$ of the kind described in  Theorem \ref{thm:V} (we shall require one such vector bundle $V_{\Satake}$ for each $V$; the only properties required of $V_{\Satake}$ are those described in the theorem). 

\begin{definition}
 We shall denote by  $\End(V)$ the pullback of the bundle  $V_{\Satake}^{\phantom{*}}\boxtimes V_{\Satake}^*$ over $\Satake\times \Satake$ along the map 
\[ (\operatorname{target},\operatorname{source}): \Groupoid_\Satake \longrightarrow  \Satake\times \Satake .
\]
We identify the fiber over $\alpha\in \Groupoid_{\Satake}$ with $\mathfrak{B}(V_{\operatorname{source}(\alpha)}, V_{\operatorname{target}(\alpha)})$ 
and equip $\End(V_{\Satake})$ with the Fell bundle structure \cite[Sec.\,2]{Kumjian98}, in which the multiplication operation between fibers over composable morphisms is composition of linear operators on the fibers, and the involution is the usual adjoint. 
\end{definition}

Form the convolution $*$-algebra $\Cc  (G;\End(V))$ of \continuous, compactly supported sections of $\End(V)$ using the above Fell bundle structure. There are source fiber representations 
\begin{equation}
\label{eq-source-fiber-reps-with-v}
\Cc  (G;\End(V)) 
\longrightarrow 
\mathfrak{B}\bigl (L^2(\Groupoid_{\Satake,x};V_{\Groupoid_{\Satake,x}})\bigr )
\end{equation}
on the Hilbert spaces of $L^2$-sections of the restriction of $V_{\Satake}$ to the source fibers of $\Groupoid_{\Satake}$, and using these we may form the $C^*$-algebra completion $C^*_r (\Groupoid_{\Satake};\End(V))$, as in \eqref{eq-reduced-groupoid-c-star-norm}.

The $C^*$-algebra $C^*_r (\Groupoid_{\Satake};\End(V))$ is Morita equivalent to $C^*_r(\Groupoid_{\Satake})$, but it is nevertheless important to us in its own right because of the existence of an integration morphism 
\begin{equation}
    \label{eq-new-integration-morphism}
C^*_r(G) \longrightarrow C^*_r(\Groupoid_{\Satake};\End(V))
\end{equation}
that differs from the original in Theorem~\ref{thm-integration-morphism-at-c-star-level} from the point of view of harmonic analysis.  The morphism \eqref{eq-new-integration-morphism} is defined first at the level of \continuous, compactly supported functions by a variation on the formula that we used Definition~\ref{def-integration-morphism}:
\begin{equation}
    \label{eq-integration-morphism-with-v}
f_! (gx,g,x) = \delta  (g,x)^{-1/2} \int _{H_x} f(gh){\cdot} u_{g,x}\, dh ,
\end{equation}
where $u_{g,x}\colon V_x\to V_{gx}$ is the unitary isomorphism between fibers of $V_{\Satake}$ given by the $G$-action on the Hermitian bundle $V_{\Satake}$.  We may then analyze this morphism almost exactly as we handled the scalar case in Sections~\ref{sec-integration-morphism} and \ref{sec-integration-morphism-at-c-star-level}.

To begin, the isomorphism of Hermitian $G$-equivariant vector bundles given  in Theorem~\ref{thm:V} gives a $G$-equivariant isomorphism of spaces of \continuous, compactly supported sections, 
\[
\Cc  (\Groupoid_{\Satake,x}; V_{\Groupoid_{\Satake,x}}) \ni \xi \longmapsto  \xi_{\natural} \in \Cc  (G/H_x) ,
\]
and using this we may define a second isomorphism
\begin{equation}
    \label{eq-new-tilde-map}
    \Cc  (\Groupoid_{\Satake,x}; V_{\Groupoid_{\Satake,x}}) \ni \xi \longmapsto  \tilde \xi \in \Cc  (G/H_x)
\end{equation}
by  means of the formula $\tilde \xi (g)  = \delta (g,x)^{1/2}{\cdot}\xi_\natural (g)$.  With this, we compute that: 

\begin{lemma} 
    \label{lem-l-2-isometric-isomorphism-with-v}
The morphism \eqref{eq-new-tilde-map}
extends to a unitary isomorphism 
\[
\pushQED{\qed}
L^2 ( \Groupoid_{\Satake,x}) \stackrel \cong \longrightarrow L^2 (G/H_x).
\qedhere \popQED
\]
\end{lemma} 

\begin{lemma} 
\label{lem-compatibility-of-regular-reps-with-v-via-integration}
For every $x \in \Satake$ there  is a commuting diagram 
\[
\xymatrix{
\Cc  (G)\ar[d]_{\lambda_{G/H_x}} \ar[r]& \Cc  (\Groupoid_{\Satake};\End(V))\ar[d]^{\lambda _x}
\\
\mathfrak{B}\bigl (L^2 (G/H_x;V)\bigr) \ar[r]_\cong & \mathfrak{B}\bigl (L^2(\Groupoid_{\Satake,x};V_{\Groupoid_{\Satake,x}})\bigr )
}
\]
in which the top arrow is the integration morphism in \eqref{eq-integration-morphism-with-v}, and the bottom arrow is induced from the inverse of the unitary isomorphism in Lemma~\ref{lem-l-2-isometric-isomorphism-with-v}. \qed
\end{lemma} 

With these in hand, the extension of the integration morphism from \eqref{eq-new-integration-morphism} to a morphism of $C^*$-algebras, and indeed the extension of the commuting diagram in  Lemma~\ref{lem-compatibility-of-regular-reps-with-v-via-integration} to the commuting diagram of $C^*$-algebra morphisms
\begin{equation}
    \label{eq-commuting-diagram-integration-and-regular-reps}
\xymatrix{
C^*_r (G)\ar[d]_{\lambda_{G/H_x}} \ar[r]& C^*_r(\Groupoid_{\Satake};\End(V))\ar[d]^{\lambda _x}
\\
\mathfrak{B}\bigl (L^2 (G/H_x;V)\bigr) \ar[r]_\cong & \mathfrak{B}\bigl (L^2(\Groupoid_{\Satake,x};V_{\Groupoid_{\Satake,x}})\bigr )
}
\end{equation}
is handled exactly as in the proof of Theorem~\ref{thm-integration-morphism-at-c-star-level}.

\section{The Harish-Chandra principle} 
\label{sec-harish-chandra-principle}

\subsection{Statement of the principle} 
\label{sec-statement-of-h-c-principle}
Here is our main result:

\begin{theorem}[Harish-Chandra principle]
\label{thm-harish-chandra-principle}
Let $G$ be a real reductive group with Iwasawa decomposition $G{=}KAN$ and corresponding system of simple restricted roots $\Sigma$, and let 
 $\pi$ be a tempered irreducible representation of  $G$.  Exactly one of the following is true:
\begin{enumerate}[\rm (i)]

\item either $\pi$ is an admissible topological discrete series representation, modulo center, in the sense of of Definition~\textup{\ref{def-tds-mod-center}}, or  

\item $\pi$ may be embedded as a  subrepresentation of some  parabolically induced representation $\Ind_{P_I}^G \sigma$, where $I$ is a proper subset of $\Sigma$ and $\sigma$ is a tempered irreducible representation of  $L_I$.
\end{enumerate}
\end{theorem}

This is basically due to Harish-Chandra, although a precise statement does not appear in his work.  For that, see \cite{Langlands89} or \cite{Trombi77}.  By applying the same principle again to the representation $\sigma$ in the statement of the theorem, and by using the principle of induction in stages \cite[(7.5)]{Knapp:representation_theory}, one obtains from the theorem the following well-known consequence:

\begin{corollary} 
Let $\pi$ tempered irreducible representation of a real reductive group $G$ with Iwasawa decomposition $G=KAN$.  There is a standard parabolic subgroup $P_I=M_IA_IN_I$ of $G$, an admissible topological discrete series representation $\sigma$ of $M_I$,  a unitary character $e^{i\nu}$ of  $A_I$, and an embedding of $\pi$ into the parabolically induced representation $\Ind_{P_I}^G\sigma \otimes e^{i\nu}$. \qed
\end{corollary}

The fact that admissibility is built into our approach to the discrete series from the beginning, as in  Theorem~\ref{thm:Spec_J},  has another interesting consequence: 

\begin{corollary}
    Every tempered irreducible unitary representation  of a real reductive group   is admissible.
\end{corollary}

\begin{proof}
    Frobenius reciprocity shows that the parabolically induced representation obtained from an admissible representation is again admissible; compare \cite[Prop.\,8.4]{Knapp:representation_theory}.
\end{proof}

This is again a well-known result of Harish-Chandra \cite[Thm.\,6]{HarishChandra53}, who proved the same for all irreducible unitary representations.  The approach to the result presented here resembles Bernstein's proof of the admissibility of smooth, irreducible representations of reductive $p$-adic groups \cite{Bernstein92notes}.

\subsection{Proof of the main theorem}

The Harish-Chandra principle, as stated in Theorem~\ref{thm-harish-chandra-principle}, is equivalent to the following identity: 
 
\begin{theorem}\label{thm-compact-is-cuspidal}
If $G$ is any real reductive group, then $C^*_{\cmc}(G)= C^*_{\cusp}(G)$.
\end{theorem}

We shall use the integration morphisms in \eqref{eq-commuting-diagram-integration-and-regular-reps} to reduce the proof of Theorem~\ref{thm-compact-is-cuspidal} to a much easier  assertion about ideals in the reduced  $C^*$-algebra of the Satake groupoid. 

\begin{definition}
    If $K$ is a finite-dimensional unitary representation of $K$, and if $I$ is any subset of $\Sigma$, then we shall write 
\[
L^2(G/H_I;V)  = [L^2 (G/N_I)\otimes V]^{K_I}.
\]
Recall that $H_I = K_IN_I$; see \eqref{eq-def-of-h-i}; in the display, we restrict the representation $V$ from $K$ to $K_I$.
\end{definition}

\begin{lemma}
\label{lem-cuspidal-ideal-and-v-isotypical-spaces}
If $G{=}KAN$ is any real reductive group, then 
    \[
    C^*_{\cusp}(G) = \bigcap_{\substack{I\subsetneqq \Sigma\\ \text{\rm $V$\,f.d.\,rep.\,of\,$K$}}}
    \operatorname{ker} \bigl (  C^*_r(G) \longrightarrow \mathfrak{B}( L^2 (G/H_I;V) \bigr )
    \]
where the intersection is over all subsets of $\Sigma$ other than $\Sigma$ itself, and over all finite-dimensional unitary representations of $K$.
\end{lemma}

\begin{proof}
    This follows from the $K_I$-isotypical decomposition 
    \[
    L^2 (G/N_I) \cong \bigoplus _{W\in \widehat {K_I}} \bigl [ L^2 (G/N_I)\otimes W \bigr ]^{K_I} \otimes W^*
    \]
    together with the fact that every irreducible unitary representation of $K_I$ occurs within some finite-dimensional unitary representation of $K$, upon restriction.
\end{proof}

\begin{lemma}
\label{lem:dense_source_rep_faithful}
The source fiber representation 
\[
\lambda_{x_\Sigma}\colon C_r^*(\Groupoid_{\Satake};V)
\longrightarrow \mathfrak{B}\bigl ( L^2(G/K;V) \bigr)
\]
is faithful.  Under the identification
\[
 L^2(G/K;V) \stackrel \cong \longrightarrow L^2(G/H_{\Sigma};V) \otimes L^2 (A_{\Sigma})
 \]
 the 
restriction of $\lambda_{x_\Sigma}$ to the ideal $C^*_r (\Groupoid_{\Interior};V)$ is an isomorphism
\[
C_r^*(\Groupoid_{\Interior};V) \xrightarrow{\cong} \mathfrak{K} (L^2(G/H_{\Sigma}; V))\otimes C^*(A_\Sigma)
\]
\end{lemma}

\begin{proof}
The source-fiber representations 
\[
\lambda_x \colon C^*_r (\Groupoid_{\Satake} ; V) \longrightarrow \mathfrak{B}(L^2 (\Groupoid_{\Satake,x}))\qquad (x\in \Satake)
\]
have the semicontinuity property
\[
\|\lambda_x(f)\| \le \liminf_{y\to x}\|\lambda_y(f)\|\qquad \forall f\in C^*_r (\Groupoid_{\Satake}).
\]
The lemma follows from this, since $\Interior$ is dense in $\Satake$.
\end{proof}

\begin{lemma}
\label{lem-integration-characterizations-of-cpt-and-cusp-ideals}
The compact modulo center and cuspidal ideals in $C^*_r(G)$ may described using the integration morphisms in \eqref{eq-commuting-diagram-integration-and-regular-reps} as follows: 
\[
C^*_{\cmc}(G) = \bigl \{\, f \in C^*_r(G) : f_! \in C^*_r(\Groupoid_{\Interior};V)\,\,\,\forall\, \text{\rm $V$\,f.d.\,rep.\,of\,$K$}  \,\bigr \} 
\]
and
\[
C^*_{\cusp}(G) = \bigl \{\, f \in C^*_r(G) : f_! \vert_{\partial \Satake} =0  \in C^*_r(\Groupoid_{\partial \Satake};V)\,\,\,\forall\, 
\text{\rm $V$\,f.d.\,rep.\,of\,$K$}  \,\bigr \} .
\]
\end{lemma}

\begin{proof}
Thanks to Lemma~\ref{lem:dense_source_rep_faithful} and the commutativity of the diagram in \eqref{eq-commuting-diagram-integration-and-regular-reps}, the assertion that $f_!\in C^*_r(\Groupoid_{\Interior};V)$ is equivalent to the assertion that under the regular representation of $C^*_r(G)$ on $L^2(G/K;V)$, $f$ maps to an element in the image of the regular representation of $C^*_r(\Groupoid_{\Interior};V)$ (we are identifying Hilbert spaces here as in Lemma~\ref{lem-l-2-isometric-isomorphism-with-v}).  But this image is precisely the $C^*$-algebra of operators described in Lemma~\ref{lem-groupoid-c-star-algebra-as-compacts}.  So by Theorem~\ref{theorem-comapct-mod-center-ideal} the assertion that $f_!\in C^*_r(\Groupoid_{\Interior};V)$ for all $V$ is equivalent to the assertion that $f\in C^*_{\cmc}(G)$, as required.

The second assertion is a consequence of the commutativity of the diagram in \eqref{eq-commuting-diagram-integration-and-regular-reps} again, together with Lemma~\ref{lem-cuspidal-ideal-and-v-isotypical-spaces} and the fact that, by definition of the norm on $C^*_r(\Groupoid_{\partial \Satake};V)$, this $C^*$-algebra is represented faithfully on the direct sum of all $L^2 (G/H_I;V)$ (so that an element of the $C^*$-algebra is zero if and only if it is zero as an operator on each Hilbert space summand).
\end{proof}

\begin{lemma} 
\label{lem-short-exact-sequence}
Let $V$ be a finite-dimensional unitary representation of $K$. The inclusion and restriction morphisms 
\[
\Cc  (\Groupoid_{\Interior};V) \to  \Cc  (\Groupoid_{\Satake};V)
\quad \text{and} \quad 
\Cc  (\Groupoid_{\Satake};V) \to  \Cc  (\Groupoid_{\partial \Satake};V)
\]
extend to the reduced groupoid $C^*$-algebra completions, and fit into a short exact sequence of $C^*$-algebras
\[
0 \to C_r^*(\Groupoid_{\Interior}; V) \to C_r^*(\Groupoid_{\Satake}; V) \to C_r^*(\Groupoid_{\partial \Satake}; V) \to 0 .
\]
\end{lemma}

\begin{proof}
This is fairly easy case of a well-known argument using the notion of amenability; see for instance \cite[Rem.\,4.10]{Renault91}. First, if $\mathfrak{Z}$ is any closed, $G$-invariant subset of $\Satake$, then the top row of the diagram
\[
\xymatrix{
0 \ar[r] & C^*_{\max} (\Groupoid_{\Satake\setminus \mathfrak{Z}}) \ar[r]\ar[d]&
C^*_{\max} (\Groupoid_{\Satake}) \ar[r]\ar[d]&
C^*_{\max} (\Groupoid_{\mathfrak{Z}}) \ar[r]\ar[d]& 0
\\
0 \ar[r] & C^*_{r} (\Groupoid_{\Satake\setminus \mathfrak{Z}}) \ar[r]&
C^*_{r} (\Groupoid_{\Satake}) \ar[r]&
C^*_{r} (\Groupoid_{\mathfrak{Z}}) \ar[r]& 0 ,
}
\]
involving full groupoid $C^*$-algebras, is exact; compare \cite[Lem.\,1]{Green77}.  Second, a diagram chase shows that the exactness of the bottom row is a consequence of the assertion that the morphism 
\[
C^*_{\max} (\Groupoid_{\mathfrak{Z}}) \longrightarrow C^*_{r} (\Groupoid_{\mathfrak{Z}})
\]
is an isomorphism.  Third, since $\Satake$ has only a finite number of $G$-orbits, this isomorphism statement may be proved by induction, using the fact that when $\mathfrak{Z}$ is a single $G$-orbit --- of a point $x_I$ as in \eqref{eq-isotropy-of-x-i} --- the groupoid 
\[
\Groupoid_{\mathfrak{Z}} \cong G / H_I\times _{A_I} G/H_I
\]
is amenable; indeed it is Morita equivalent to the abelian group $A_I$.  Compare \cite[Prop.\,6.1.8]{AnantharamanDelarocheRenault}.  We have written the argument for the scalar case, but the Fell bundle case is exactly the same.
\end{proof}

\begin{proof}[Proof of Theorem~\ref{thm-compact-is-cuspidal}]
In view of the formulas in Lemma~\ref{lem-integration-characterizations-of-cpt-and-cusp-ideals}, the equality of $C^*_{\cmc}(G)$ and $C^*_{\cusp}(G)$ is a consequence of the equality 
\[
C_r^*(\Groupoid_{\Interior}; V) = \{\, f\in C^*_r(\Groupoid;V) : f\vert_{\partial \Satake} = 0\,\} ,
\]
which follows from Lemma~\ref{lem-short-exact-sequence}.
\end{proof}

\appendix

\section{Proof of Theorem \ref{thm:V}}
\label{app-vector-bundle}

Let $G=KAN$ be a real reductive group with Cartan involution $\theta$.  In this appendix, we shall associate to each finite-dimensional unitary representation  of $K$  a $G$-equivariant Hermitian vector bundle   on the Satake compactification with the property described in Theorem~\ref{thm:V}.

\subsection{Oshima's construction of the Satake compactification} 
Oshima \cite{Oshima78} constructed a smooth, closed $G$-manifold $\Oshima$ and an inclusion of the Satake compactification $\Satake$ into $\Oshima$ as a closed subset. 
We refer the reader to \cite[Sec.\,3]{BraddHigsonYunckenOshimaPart1} for a detailed exposition using the same notation that we shall use below, but in brief, if $\Sigma\subseteq \Hom_{\R}(\mathfrak{a},\R)$ is the set of simple restricted roots associated to the given Iwasawa decomposition of $G$, and if $\R^\Sigma$ is the space of real-valued finite sequences indexed by $\Sigma$, then $\Oshima$ is constructed in two stages as a quotient of $\bigG= G{\times}\R^\Sigma$ by an equivalence relation. In the first stage a smooth family  
\[
\bigH = \{\, (h,t) \in \bigG : h\in H_t  \,\}
\]
of closed subgroups  $H_t\subseteq G$ indexed by $\R^\Sigma$ is constructed, and the associated smooth family of homogeneous spaces $\bigG/\bigH$ is formed.   In the second stage, an action of $A$ on $\R^\Sigma$ is defined by 
\[
(a \cdot t )_\alpha = e^{\alpha (\log (a))} t_\alpha\qquad \forall \alpha \in \Sigma .
\]
The groups $H_t$ have the property that 
\[
\Ad_{a} H_t = H_{a\cdot t}\qquad \forall a\in A,\,\,\, \forall t\in \R^\Sigma ,
\]
and so a right-action of $A$ on $\bigG/\bigH$ may be defined by
\[
(gH_t,t)\cdot a = (gH_ta, a^{-1}\cdot t) = (gaH_{a^{-1}\cdot t}, a^{-1}\cdot t)\qquad \forall g\in G,\,\,\, \forall a\in A,\,\,\, \forall t\in \R^\Sigma .
\]
The Oshima space is the quotient
\[
\Oshima = (\bigG/\bigH) / A .
\]
It carries a unique smooth $G$-manifold structure for which the quotient map from $\bigG$ (with the left-translation action of $G$) is a $G$-equivariant submersion. See \cite[Thm.\,3.3.12]{BraddHigsonYunckenOshimaPart1}.

Thinking of $\Oshima$ as a quotient of $\bigG$ by an equivalence relation, the  Satake compactification is the image in $\Oshima$ of the space 
$\bigG_{\nonnegative}  = G {\times}\R^\Sigma_{\spacenonnegative}$, 
where $\R^\Sigma_{\spacenonnegative}\subseteq \R^\Sigma$ is the space of sequences with zero or positive entries.  We may therefore  describe the Satake compactification as  
\[
\Satake = (\bigG_{\nonnegative} / \bigH_{\nonnegative})/A . 
\]

\subsection{A continuous family of finite-dimensional unitary representations} 
If $I$ is a subset of $\Sigma$, let us define $t_I\in \R^\Sigma_{\spacenonnegative}$ by
\begin{equation}
    \label{eq-def-of-t-i}
t_{I,\alpha} = 
\begin{cases}
    1 & \alpha \in I \\
    0 & \alpha \notin I .
\end{cases}
\end{equation}
The group $ H_{t_I}$ has a semidirect product structure 
\[
H_{t_I} = K_I\ltimes \theta (N_I) ,
\]
where $K_I\subseteq K$ and  $N_I\subseteq N$; see after \cite[Def. 3.2.20]{BraddHigsonYunckenOshimaPart1}. So $K_I$ is a quotient of $H_I$.

\begin{lemma}
\label{lem-smooth-family-of-fd-unitary-reps}
Let $\pi\colon K \to U(V)$ be a finite-dimensional unitary representation of $K$
    There is a unique family of continuous unitary representations 
    \[
    \pi_t\colon H_t \longrightarrow U(V)\qquad (t\in \R^\Sigma_{\spacenonnegative})
    \]
    such that 
    \begin{enumerate}[\rm (i)]

        \item for every $I\subseteq \Sigma$, the representation $\pi_{t_I}$ is the composition 
        \[
        H_I\stackrel{\text{\rm quot.}}\longrightarrow K_I \stackrel{\text{\rm incl.}}\longrightarrow K \stackrel{\pi} \longrightarrow U(V) ,
        \]
        \item if $t\in \R^\Sigma_{\spacenonnegative}$,  $h\in H_t$ and $a\in A$, then  $\pi_t(h) = \pi_{a\cdot t}(aha^{-1})$, and 

        \item the map
        \[
        \bigH_{\nonnegative} \ni (h,t) \longmapsto \pi_t(h) \in U(V)
        \]
        is continuous.
    \end{enumerate}
\end{lemma}

\begin{proof}
The elements $t_I\in \R^\Sigma_{\spacenonnegative}$ constitute a complete set of orbit representatives for the $A$-action on $\R^\Sigma_{\spacenonnegative}$, so the family $\{\pi_t\}$ is certainly uniquely determined by (i) and (ii), assuming it exists. 

As for existence, given $t\in \R^\Sigma_{\spacenonnegative}$, choose $a\in A$ and $I\subseteq \Sigma$ such that $t= a\cdot t_I$, and then define
\[
\pi_t (h) = \pi_{t_I}(a^{-1} ha)\qquad \forall h\in H_t ,
\]
where $\pi_{t_I}$ is defined as in (i). Now the Lie algebra of $H_t$ has the form
\[
      \lie{h}_t = \lie{m} \oplus \bigoplus_{\gamma\in\Delta^+(\mathfrak{g},\mathfrak{a})} \bigl \{ \,  t^{2 \gamma}X {+} \theta(X) : X_\gamma\in \lie{g}_\gamma\, \bigr \} 
    ,
\]  
where the direct sum is over the positive restricted roots, and $\lie{m}$ is the Lie algebra of the centralizer of $\mathfrak{a}$ in $K$. See \cite[Def.\,3.2.7]{BraddHigsonYunckenOshimaPart1}.  One  computes that the Lie algebra representation of $\lie{h}_t$ associated to $\pi_t$ has the form 
\[
\begin{gathered}    
d\pi_t\big \vert_{\lie{m}} = d\pi\big \vert_{\lie{m}}
\\
d\pi_t : t^{2\gamma}X_\gamma +\theta(X_\gamma) \mapsto t^\gamma d\pi\bigl (X_\gamma +\theta(X_\gamma)\bigr ),
\end{gathered}
\]
where $t^{\gamma}=\prod _{\alpha\in \Sigma} t_\alpha^{n_{\gamma,\alpha}} $, with the nonnegative integers $n_{\gamma,\alpha}$ being the coefficients of the expression $\gamma = \sum _{\alpha\in \Sigma} n_{\gamma,\alpha} \alpha$,  as in \cite[Def.\,3.2.5]{BraddHigsonYunckenOshimaPart1}. 
This shows that $\pi_t$ is independent of the choices made in its definition, and that the family  $\{\pi_t\}$ is continuous in $t\in \R^\Sigma_{\spacenonnegative}$, at least  at the Lie algebra level.  

Now,  every element of the group $H_t$ is a product of an element of  $M$ and an element of the connected Lie subgroup with Lie algebra $\lie{h}_t$, see \cite[Def.\,3.2.20]{BraddHigsonYunckenOshimaPart1}. Since  $\pi_t \vert_M = \pi\vert_M$, for all $t$, the family of Lie group representations  $\{\pi_t\}$ is continuous at the group level, as in item (iii) above.  This completes the proof. 
\end{proof}

\subsection{Construction of the vector bundle}
We are now ready to construct our $G$-equivariant Hermitian vector bundles on $\Satake$.

\begin{proof}[Proof of Theorem \ref{thm:V}]
    We shall construct the vector bundle in the same way that Oshima constructed the Satake compactification.  Thus we begin by forming the topological space $(\bigG_{\nonnegative}{\times}V) / \bigH_{\nonnegative}$, which is the quotient of $\bigG_{\nonnegative}{\times}V$ by the equivalence relation
    \[
    (g,t,v) \sim (gh,t,\pi_t(h^{-1})v) \qquad \forall g\in G,\,\,\,\forall t\in \R^\Sigma_{\spacenonnegative},\,\,\, \forall h\in H_t ,\,\,\, \forall v\in V.
    \]
    Then we equip the quotient space with the $A$-action 
    \[
    [g,t,v]\cdot a = [ga,a^{-1}t, v]
    \]
    and form the further topological quotient space 
    \[
    V_{\Satake} = ((\bigG_{\nonnegative}{\times}V) / \bigH_{\nonnegative})/A.
    \]
    This is a family of vector spaces over the Satake compactification, in the sense of \cite[Sec.\,1.1]{AtiyahKTheory}, and the left-translation action of $G$ preserves this structure. Local triviality of the family may be checked  using the coordinate map 
    \[
    \begin{gathered}
    N\times \R^\Sigma_{\spacenonnegative} \longrightarrow \Satake
    \\
    (n,t)\longmapsto \ldoublebracket n,t\rdoublebracket 
    \end{gathered}
    \]
    (the double brackets indicate equivalence classes) which is a homeomorphism onto an open subset of $\Satake$; see \cite[Lem.\,3.4.3]{BraddHigsonYunckenOshimaPart1}. The inner product on $V$ gives the equivariant vector bundle we have just defined a $G$-invariant Hermitian structure.  Finally, if $I\subseteq \Sigma$ and if $t_I\in \R^\Sigma_{\spacenonnegative}$ is as in \eqref{eq-def-of-t-i}, then the map
    \[
    \begin{aligned}
     V_{\Satake}\big \vert_{\Satake_I}  & \longrightarrow G/A_IH_I\times_{K_I} V 
     \\
       \ldoublebracket g,at_I,v\rdoublebracket& \longmapsto \ldoublebracket ga,v\rdoublebracket 
     \end{aligned}
    \]
     is a $G$-equivariant vector bundle isomorphism  over  $\Satake_I\cong G/A_IH_I$, as required.
\end{proof}

\subsection{Remark about vector bundles on the Oshima space}
One might ask if the construction in the previous section can be improved  so as to produce a \emph{smooth}, $G$-equivariant Hermitian vector bundle on the Oshima space $\Oshima$?  The answer is yes in a range of cases.

One approach is to examine the family of Lie algebra representations  
\[
\lie{h}_t \ni  t^{2\gamma}X_\gamma +\theta(X_\gamma) \stackrel{d\pi_t} \longmapsto t^\gamma d\pi\bigl(X_\gamma +\theta(X_\gamma)\bigr)\in \mathfrak{u}(V)
\]
with $d\pi_t\big \vert_{\lie{m}} = d\pi\big \vert_{\lie{m}}$  that was encountered in the proof Lemma~\ref{lem-smooth-family-of-fd-unitary-reps}. It extends using the same formula  to a smooth family of infinitesimally unitary representations $\pi_t \colon \lie{h}_t \to \lie{u}(V)$ parametrized by $t\in \R^\Sigma$. If each of these Lie algebra representations integrates to a Lie group representation   $\pi_t \colon H_t \to U(V)$ with $\pi_{t}\big \vert_{M} = \pi\big\vert _M$, then the arguments in the previous two sections will construct a smooth, $G$-equivariant  Hermitian vector bundle on $\Oshima$.

The Lie algebra representations $\pi_t$ do not always integrate to Lie group representations; counterexamples can be found among the finite covers of the group  $SL(2,\R){\times}SL(2,\R)$, or even of $SO(2){\times}SL(2,\R)$.  But if $G$ has a complexification $G_{\C}$, then this problem does not arise. Indeed,  for every $t\in \R^\Sigma$ there is an element of $A_{\C}$ that conjugates   $H_t$ into itself and 
conjugates the Lie algebra representation $d\pi_t$ to the representation $d\pi_{|t|}$.  So integrability of $\pi_t$ follows from the proof of Lemma~\ref{lem-smooth-family-of-fd-unitary-reps}. 
 
As a result, if  $G$ is a linear  real reductive group, then each of the $G$-equi\-variant Hermitian vector bundles $V_{\Satake}$ constructed in the proof of Theorem~\ref{thm:V} is the restriction to $\Satake$ of a smooth, $G$-equivariant Hermitian vector bundle over the Oshima space $\Oshima$.

\bibliographystyle{alpha}
\bibliography{refs}

\end{document}